\documentclass[letterpaper, 10 pt, conference]{ieeeconf}  % Comment this line out if you need a4paper

\IEEEoverridecommandlockouts                              % This command is only needed if 
\usepackage{amsmath}
\usepackage{amsfonts}
\usepackage{amssymb}

\usepackage{amsthm}
\usepackage{graphicx}
\usepackage[colorlinks=true, allcolors=blue]{hyperref}

\usepackage{caption}
\usepackage{subcaption}
\usepackage{dsfont}

\usepackage{placeins}

\usepackage{xurl}

\newcommand{\y}{\mathbf{y}}
\newcommand{\bc}{\mathbf{c}}
\newcommand{\bchat}{\mathbf{\hat{c}}}
\newcommand{\chat}{\hat{c}}

\DeclareMathOperator*{\argmax}{arg\,max}

\newcommand{\subS}{{}_{{\scalebox{0.5}{$\mathrm{S}$}}}}
\newcommand{\subN}{{}_{{\scalebox{0.5}{$\mathrm{N}$}}}}
\newcommand{\subI}{{}_{{\scalebox{0.5}{$\mathrm{I}$}}}}
\newcommand{\subR}{{}_{{\scalebox{0.5}{$\mathrm{R}$}}}}
\newcommand{\subD}{{}_{{\scalebox{0.5}{$\mathrm{D}$}}}}

\newcommand{\cS}{c\subS^{}}
\newcommand{\cSS}{c\subS^{*}}
\newcommand{\cNS}{c\subN^{*}}

\newcommand{\cI}{c\subI^{}}
\newcommand{\cR}{c\subR^{}}
\newcommand{\cN}{c\subN^{}}

\newcommand{\cIbar}{\bar{c}\subI^{}}

\newcommand{\cNbar}{\bar{c}\subN^{}}

\newcommand{\uI}{u\subI^{}}

\newcommand{\uN}{u\subN^{}}

\newcommand{\uIbar}{\bar{u}\subI^{}}

\newcommand{\uNbar}{\bar{u}\subN^{}}

\newcommand{\VS}{V_{{\scalebox{0.5}{$\mathrm{S}$}}}}
\newcommand{\VI}{V_{{\scalebox{0.5}{$\mathrm{I}$}}}}
\newcommand{\VR}{V_{{\scalebox{0.5}{$\mathrm{R}$}}}}
\newcommand{\VN}{V_{{\scalebox{0.5}{$\mathrm{N}$}}}}
\newcommand{\VD}{\Psi_{{\scalebox{0.5}{$\mathrm{D}$}}}}

\newcommand{\newest}{\textcolor{black}}

\usepackage[noadjust]{cite}

\newtheorem*{theorem*}{Theorem}

\newtheorem*{example*}{Example}

\newtheorem*{fact*}{Fact}

\newtheorem*{prop*}{Proposition}

\newtheorem*{rem*}{Remark}

\newtheorem*{assumption*}{Assumption}

\title{\LARGE \bf
Partial health status observability and time horizon uncertainty\\ in mean-field game epidemiological models$^*$
}

\author{Carlos Doebeli$^{1}$ and Alexander Vladimirsky$^{2}$% <-this % stops a space
\thanks{*This work was partially supported by 
the NSF (DMS-2111522),
AFOSR (FA9550-22-1-0528), 
and the Royal Society Wolfson Visiting Fellowship awarded to the second author.}% <-this % stops a space
\thanks{$^{1}$Carlos Doebeli is with the Department of Mathematics, Imperial College London, SW7 2AZ, UK,
        {\tt\small c.doebeli22@imperial.ac.uk}}%
\thanks{$^{2}$Alexander Vladimirsky is with the Department of Mathematics, Cornell University,
        Ithaca, NY 14850, USA
        {\tt\small vladimirsky@cornell.edu}}%
}

\begin{document}

\maketitle
\thispagestyle{empty}
\pagestyle{empty}

%%%%%%%%%%%%%%%%%%%%%%%%%%%%%%%%%%%%%%%%%%%%%%%%%%%%%%%%%%%%%%%%%%%%%%%%%%%%%%%%
\begin{abstract}
We introduce Mean-Field Game (MFG) epidemiological models, in which 
immunity either wanes with time in a fully observable way
or disappears instantaneously with no direct observation
(making a previously recovered individual fully susceptible again 
without realizing it).
Both interpretations create 
computational challenges for rational noninfected individuals 
deciding on their contact rates based on their personal current immunity 
state and the changing epidemiological situation.  
Both require solving a forward-backward MFG system that includes PDEs
(an advection-reaction equation for the immunity-structured population and a Hamilton-Jacobi-Bellman equation for the corresponding value function).
We show how this can be done efficiently by solving a two-point boundary value problem for a system of approximating ODEs.  We also show how the same approach can be extended to handle an initial uncertainty in the planning horizon.
\end{abstract}

%%%%%%%%%%%%%%%%%%%%%%%%%%%%%%%%%%%%%%%%%%%%%%%%%%%%%%%%%%%%%%%%%%%%%%%%%%%%%%%%
\section{INTRODUCTION}

Details of human behavior are an important factor in the spread of infectious diseases.
While in traditional epidemiological models most types of behavior (e.g., individuals' degree of compliance with social distancing recommendations) are pre-programmed, in Mean-Field Game (MFG) models individuals are assumed to make decisions (e.g., on their individual contact rate) rationally 
to maximize their personal ``payoff' 
based on their current health status and the evolving epidemiological situation
\cite{reluga2010game, tembine2020covid, doncel2022mean, aurell2022optimal, olmez2022modeling, aurell2022finite, roy2023recent, bremaud2024mean, bremaud2025mean, buckley2025behavioral, liu2025incorporating}.  
Here, our goal is to show how such models can be extended to handle partial observability of health states and uncertainty in the planning horizon.
Both of these features have obvious practical value since individuals
often do not fully know their immunity status and the duration of the epidemic is \emph{a priori} unknown.

A common assumption in many epidemiological models is that an individual gains immunity after becoming infected and then recovering.  How long that immunity lasts depends on the disease, but in most cases becoming susceptible is viewed as an instantaneous randomly-timed event, and an individual is assumed to be fully aware when this happens.  (As we explain in \S\ref{sec:basic_MFG}, this awareness is important in traditional MFG models, since it can be used to decide on reducing the contact rates to lower the chances of re-infection.)  Here, we consider and compare two more realistic modeling approaches: 
immunity that gradually wanes in a fully observable way (\S\ref{sec:full_obs_wane})
versus full immunity that instantaneously disappears at some random (and unobserved) moment (\S\ref{sec:unobs_disap}).   The latter version can be interpreted in the framework of {\em MFGs over belief space}, which was previously also used (though with many restrictive assumptions) in \cite{olmez2022modeling} to model the behavior of presymptomatic infected individuals.  
Both of our nuanced immunity models lead to MFG systems that include PDEs
(due to an immunity spectrum).  In \S\ref{sec:numerics} we show how these
models can be approximated by a two point boundary value problem for an extended system of ODEs.  Moreover, we show that the version of this problem with an uncertain planning horizon (\S\ref{sec:uncertain}) can be solved in a similar way but with additional jump conditions for each possible value of the time horizon.  Our approach is illustrated by computational experiments in \S\ref{sec:experiments}, and some directions for future work are described in \S\ref{sec:conclusions}.

\subsection*{Basic SIRSD model} \label{sec:SIRSD}

While our approach 
is quite general, 
here we describe it in the context of a specific simple Susceptible-Infected-Susceptible-Recovered-Dead (SIRSD) model,
summarized in ODEs (\ref{eq:SIRSD_ODEs_original_S}-\ref{eq:SIRSD_ODEs_original_D}) below.
For the sake of simplicity, we ignore births and disease-unrelated deaths.  
The population is split into corresponding compartments ($S$, $I$, $R$, and $D$), with the same capital letters also being used to denote their fractions of the initial population, so that $S(t)+I(t)+R(t)+D(t)=1$ at all times.
We assume a {\em density-based force of infection}, which makes the rate of adding newly-infected individuals proportional to $S$ and $I$.  
Specifically, we assume that each susceptible individual has a probability of becoming infected proportional to their contact rate $c\subS$, the contact rate of infected individuals $c\subI$, the current fraction of infected individuals $I(t),$ and a disease-specific parameter $\beta,$ which indicates 
the transmission probability
per $S$-$I$ contact.  (For now, we view all contact rates as known constants, but in the next sections they will be treated as control variables independently tuned by individuals to maximize their personal expected payoffs as the epidemic trajectory changes.)   Here, we further assume that on average an Infected individual recovers in $(1/\mu)$ days, immediately becomes fully immune upon recovery, and later the immunity disappears instantaneously (an $R\rightarrow S$ transition) in, on average, ($1/\gamma$) days after recovery.  
In addition, all infected individuals are also subject to per-capita death rate $\delta.$
\begin{align}
\label{eq:SIRSD_ODEs_original_S}
S' &\; = \; \overbrace{- \beta \cI \cS I S}^{\text{new infections}} \, + \, \overbrace{\gamma R,}^{\text{disappearing immunity}}\\[0.1em]
\label{eq:SIRSD_ODEs_original_I}
I' &\; = \;  \overbrace{\beta \cI \cS I S}^{\text{new infections}} \, - \, \overbrace{\mu I}^{\text{recoveries}} \, - \, \overbrace{\delta I,}^{\text{new deaths}}%\\[0.1em]
\end{align}
\begin{align}
\label{eq:SIRSD_ODEs_original_R}
R' &\; = \;  \overbrace{\mu I}^{\text{recoveries}} \, - \, \overbrace{\gamma R,}^{\text{disappearing immunity}}\\[0.1em]
\label{eq:SIRSD_ODEs_original_D}
D' &\; = \;  \overbrace{\delta I,}^{\text{new deaths}}\\[0.1em]
\nonumber
\text{with }& S(0) = S_0, \; I(0) = I_0, \; R(0) = R_0, \; D(0) = D_0.
\end{align}

\section{MFG-STYLE MODELS}

\subsection{Fully observable absolute immunity model (MFG-SIRSD)}

\label{sec:basic_MFG}
We now introduce a mean-field game model, which arises naturally from the assumption that each agent changes their personal contact rate selfishly and independently based on their knowledge of the current state of the epidemic. 
What determines all individual contact rates is one of the central questions that stimulated the development of behavior-epidemiological models. In a game-theoretic framework, all individuals are assumed to be acting fully rationally and independently, each striving to maximize their personal {\em expected payoff}.
In MFGs, individuals interact with a changing density of other ``players'' and interpret that personal payoff cumulatively \cite{lasry2006jeux_2, huang2006large}.
In our context, this means that the payoff includes
a time integral of some {\em utility function},
which depends on that individual's 
contact rate and changing health status. 
Following \cite{cho2020} and \cite{buckley2025behavioral}, we adopt concave utility functions
\begin{equation}
    \label{eq:utility}
    u_z(c) = \left( b_z c - c^2 \right)^g - a_z,
\end{equation}
in which $c$ is the chosen contact rate, and $z$ indicates the individual's (randomly changing) health status:
$z \in \{N, I\}.$ Here, $N$ stands for 
``Noninfected'',
which includes both Susceptible and Recovered, while $I$ stands for ``Infected''.  
We use parameter values $g = 0.25, b\subN = 10, b\subI = 6, a\subN = 0, a\subI = 4$. 
This ensures that the utility is generally lower for those Infected, and
$\cIbar = \argmax_c \uI(c) = 3$ is lower than $\cNbar = \argmax_c \uN(c) = 5.$
See \cite[Section 2.2]{cho2020} for a detailed modeling justification.
We will also use the notation $\uIbar = \uI(\cIbar)$  and $\uNbar = \uN(\cNbar)$
to denote the maximal instantaneous utility that these individuals can achieve if 
they are not concerned with the consequences of their actions for their future personal payoffs.

If the initial distribution of people among the health states $\y_0 = \left( S_0, I_0, R_0, D_0 \right)$ is known and the entire population follows some fixed contact strategy 
$\bc(t) = \left( c\subS(t), c\subI(t), c\subR(t) \right)$, the ODEs (\ref{eq:SIRSD_ODEs_original_S}-\ref{eq:SIRSD_ODEs_original_D})
yield the resulting epidemic trajectory $\y(t) = \left( S(t), I(t), R(t), D(t) \right)$.
An individual might want to change their personal contact rate strategy to
some $\bchat(t) = \left( \chat\subS(t), \chat\subI(t), \chat\subR(t) \right)$
in order to improve their personal payoff up to some planning horizon $T$.  
Assuming that the rest of the population stays with 
$\bc$, this personal choice will not affect $\y.$
Starting from a state $z(t) = \xi \in \{S, I, R, D\}$ at a time $t < T$, 
an individual's personal payoff can be calculated up to time $T$ as
%\vspace*{-8mm}
\begin{equation}
    \mathcal{J} \left( \xi, t, \bchat(\cdot); \, \y(\cdot), \bc(\cdot) \right) \! = 
     \mathbb{E}
    \left[\int_t^T
    \hspace*{-3mm} 
    u_{z(r)} \left( \chat_{z(r)} (r) \right) dr 
    +
    \Psi_{z(T)} 
    \right]\!,
\end{equation}
where $z(r)\in \{S, I, R, D\}$ is their (randomly changing) health state at a time $r \in [t,T]$
and $\Psi$ is a terminal reward or penalty.
The event of death is handled by imposing a single large penalty (i.e., $u\subD \equiv 0$ 
but $\Psi\subD$ is a large negative number).  Our interpretation of the planning horizon determines 
the other terminal penalties.  Throughout this paper, 
we assume that an effective vaccine is developed and distributed at the
time $T$, so $\Psi\subS = \Psi\subR = 0$.  If all those still infected at $t=T$ are instantaneously cured,
this implies $\Psi\subI = 0$.  (This is the approach taken in most prior MFG models, including \cite{cho2020} and \cite{buckley2025behavioral}.) In contrast, we assume that they still recover with a rate $\mu$ and die with a rate $\delta$. The average number of days that a person stays infected is therefore $1 / (\mu + \delta)$, during which time they would continue receiving a lower utility penalty $(\uIbar - \uNbar)$. 
In addition, there is also a probability $\delta / (\mu + \delta)$ that they will die before recovering, incurring a ``death penalty'' $\Psi\subD$.
The terminal penalty for $z(T) = I$ is thus 
\begin{equation} \label{eq:infected_terminal}
    \Psi\subI = \frac{\uIbar - \uNbar}{\mu + \delta} + \VD \frac{\delta}{\mu + \delta}.
\end{equation}

When an individual chooses $\bchat$ to maximize $\mathcal{J}$, 
it is worth noting that
whenever $z(r) \in \{I, R\},$ the contact rate does not affect the chances of their next health status change. 
In such states, it is thus optimal to maximize the instantaneous utility, taking $\chat\subI = \cIbar$ and $\chat\subR = \cNbar.$  On the other hand, for a susceptible individual, the optimal $\chat\subS$ is usually lower than $\cNbar,$ representing a compromise between the instantaneous utility and their fear of a possible future penalty due to infection or death. 

The {\em value function} encodes an individual's optimal expected payoff remaining up to the time $T$. More precisely, starting from the health status $\xi$ at the time $t,$ 
we define\\
$
    V_{\xi}(t) = \sup_{\bchat(\cdot)} \mathcal{J} \left(\xi, t, \bchat(\cdot); \, \y(\cdot), \bc(\cdot) \right).
$
Based on these definitions, it is obvious that $V\subD(t) = \VD.$  For all other health states, 
we use Bellman's optimality principle to write the value functions at a time $t$ in terms of the value functions at time $t + \tau$:
\begin{equation}
    \begin{aligned}
        \VS(t) =  & \max_{\chat\subS(\cdot)} 
        \Bigg\{ \left( \int_t^{t+\tau} \uN(\chat\subS(r)) dr \right) \\
        & + \left(\beta \cI(t) I(t) \chat\subS(t) \tau \right) \VI(t + \tau) \\ 
        & + \left( 1 - \beta \cI(t) I(t) \chat\subS(t) \tau \right) \VS(t + \tau) \Bigg\} + o(\tau),
    \end{aligned}
\end{equation}
\begin{equation}
    \begin{aligned}
        \VI(t) = 
        &\int_{t}^{t+\tau} \uIbar dr + \mu \tau \VR (t+\tau) + \delta \tau \VD + \\
        & \left(1 - \mu \tau - \delta \tau \right) \VI(t + \tau) + o(\tau),
    \end{aligned}
\end{equation}
\begin{equation}
    \begin{aligned}
        \VR(t) = 
        & \int_t^{t+\tau} \uNbar dr + \gamma \tau \VS(t + \tau) \\
        & + (1 - \gamma \tau) \VR(t + \tau) + o(\tau).
    \end{aligned}
\end{equation}

A standard argument based on Taylor-series expanding the above in $\tau$ yields the system of 
Hamilton-Jacobi ODEs for the value functions:
\begin{align}
\label{eq:MFG_ODEs_original_VS}
\VS' & \; = \; -\uN(\cSS) \, + \, 
\beta \cI \cSS I (\VS-\VI),\\
\label{eq:MFG_ODEs_original_VI}
\VI' & \; = \; -\uIbar \, + \, \mu\big(\VI - \VR \big) 
\, + \, \delta \big(\VI - \VD\big),\\
\label{eq:MFG_ODEs_original_VR}
\VR' & \; = \; -\uNbar \, + \, \gamma\big(\VR - \VS\big),
\end{align}
with the Nash contact rate of Susceptibles specified by
\begin{equation}
\label{eq:cSS}
\cSS \; = \; \argmax_{c}\left\{
\uN(c) \, +\,  \beta \cI c I \big( \VI - \VS \big)
\right\},
\end{equation}
and the terminal conditions $$\VS(T) = \Psi\subS, \, \VI(T) = \Psi\subI, \, \VR(T) = \Psi\subR.$$

Recall that until now, the all-but-focal-individual contact rate policy 
$\bc(t) = \left( c\subS(t), c\subI(t), c\subR(t) \right)$ was assumed fixed.
When everyone instead chooses their contact rates simultaneously and independently, their choices affect the epidemic trajectory $\y$ and the concept of optimality is replaced by {\em Nash equilibrium}.  In MFGs, a contact rate policy $\bc^*(\cdot)$
and the trajectory $\y^*(\cdot)$ form a Nash equilibrium pair if\\ (1) $\y^*(\cdot)$ is generated by ODEs
(\ref{eq:SIRSD_ODEs_original_S}-\ref{eq:SIRSD_ODEs_original_D}) using $\bc^*(\cdot)$ and\\ (2) if no individual has any incentive to change their personal contact rate strategy; i.e.,
$$
\mathcal{J} \left( \xi, t, \bc^*(\cdot); \, \y^*(\cdot), \bc^*(\cdot) \right) 
\; \geq \; 
\mathcal{J} \left( \xi, t, \bchat(\cdot); \, \y^*(\cdot), \bc^*(\cdot) \right),
$$
for all $\xi, \, t, $ and $\bchat(\cdot).$
This has the effect of coupling the ODE systems for the epidemic trajectory
(\ref{eq:SIRSD_ODEs_original_S}-\ref{eq:SIRSD_ODEs_original_D})
and the value functions
(\ref{eq:MFG_ODEs_original_VS}-\ref{eq:MFG_ODEs_original_VR}),
with the substitution of $\cI = \cIbar, \, \cS = \cSS$ in all of these and in \eqref{eq:cSS}.
We refer to this forward-backward system of seven ODEs as the {\em MFG-SIRSD model}.

Until now, we have assumed that the full immunity conferred by an $I\rightarrow R$ transition
was later instantaneously lost as a result of a fully observed  $R\rightarrow S$ transition.
This assumption is very unrealistic, but also  
quite common
in epidemiological models.
There are two very different ways of relaxing it, and we describe both of them in the next two sections.  
In both cases, the discrete $S(t)$ and $R(t)$ are replaced by a continuous spectrum of 
noninfected individuals $N(p,t),$ where $p \in[0,1]$ can be interpreted as their level of 
remaining immunity or 
the level of immunity confidence.  
In these models, $\int_0^1 N(p,t) dp \, + \, I(t) \,+\, D(t) = 1$ for all $t \in [0,T]$
and initial conditions $S(0) = S_0, \, R(0) = R_0$ are replaced by an initial distribution $N(p,0).$

\subsection{Fully observable waning immunity model} \label{sec:full_obs_wane} 
Here we suppose that one's immunity does not disappear instantaneously but instead wanes 
with time (with each noninfected individual fully aware of their current immunity level $p$),
and the transmission probability in each encounter with any already infected individual 
is decreased by a factor of $(1-p).$
For simplicity, we assume that the immunity level is absolute ($p=1$) at recovery and decreases exponentially from that point, i.e., with $p' = - \gamma p,$ until a possible re-infection.
Without re-infection, a recovery confers the ``overall immunity'' of $\int_0^{\infty} e^{-\gamma t} dt = 1/\gamma,$
which is consistent with the expected overall immunity in the original model (the full immunity with $p=1$ for the expected time of $1/\gamma$). 
This predictable decrease in immunity yields a convection-type term in a PDE for $N$
with the 
``immunity drift''\footnote{Assuming possible dependence of $f$ on the chosen contact rate $c$ and time $t$ is not needed here but becomes useful in the next subsection.} $f(p,c,t) = - \gamma p.$
Finally, since the noninfected population is $p$-structured, the total number of newly infected per unit time is computed by integrating over all $p$, while the new recoveries yield the boundary condition for $N(p=1,t)$.  
The resulting MFG system is summarized  
in equations 
(\ref{eq:NID_N}-\ref{eq:NID_cSS}).
\begin{align}
\label{eq:NID_N}
\frac{\partial N}{\partial t} &\; = \; 
\overbrace{- \frac{\partial}{\partial p} \left[ f \left(p,\cNS, t \right) N\right]}^{\text{waning immunity}} \, - \,
\overbrace{ \beta (1-p) \cIbar \cNS I N,}^{\text{new infections}}\\[0.3em]
\label{eq:NID_I}
\frac{dI}{dt} &\; = \;  \overbrace{\int_0^1 \beta (1-p) \cIbar \cNS(p,t) I N \, dp}^{\text{new infections}} \, -  
\hspace*{-2mm} \overbrace{\mu I}^{\text{recoveries}} 
\hspace*{-2mm} - 
\overbrace{\delta I,}^{\text{deaths}}\\[0.3em]
\label{eq:NID_D}
\frac{dD}{dt} &\; = \;  \overbrace{\delta I,}^{\text{deaths}}\\[0.3em]
\label{eq:NID_VN}
\nonumber
\frac{\partial \VN}{\partial t} & \; = \; -\uN\left( \cNS(p) \right) 
\, - \, f \left(p,\cNS, t \right) \frac{\partial \VN}{\partial p}\\
& \hspace*{5mm}
\, + \, 
\beta (1-p) \cIbar \cNS I (\VN-\VI),\\
\label{eq:NID_VI}
\frac{d \VI}{dt} & \; = \; -\uIbar \, + \, \mu\big(\VI - \VN(1,t) \big) 
\, + \, \delta \big(\VI - \VD\big),
\end{align}
with the boundary condition due to recoveries
\begin{equation}
\label{eq:NID_boundary}
N(1,t) \; = \; \mu I(t), \quad \text{for } t \in (0, T],
\end{equation}
and the Nash-optimal contact rate of noninfected individuals 
\begin{align}
\nonumber
& \hspace*{-3mm}
\cNS(p,t)  =  \argmax_{c}\bigg\{
\uN(c) \, +  \, f(p,c,t) \frac{\partial \VN}{\partial p}  \, +\\  
\label{eq:NID_cSS}
& \hspace*{37mm}
\beta \cIbar c I (1-p) \big( \VI - \VN \big)
\bigg\}.
\end{align}
The initial health state of the population provides the initial conditions for $N, I,$ and $D$
while the terminal penalties specify the terminal conditions $\VN(p,T) = \Psi\subN = 0$ and $\VI(T) = \Psi\subI,$
as defined in the previous section.
\newest{We note that the value function $\VN$ need not be a smooth function of $p$ and thus has to be interpreted as a {\em viscosity solution} of PDE \eqref{eq:NID_VN}; see \cite{cardaliaguet2010notes}.}

\subsection{Unobserved disappearing immunity model} \label{sec:unobs_disap}

An alternative approach to generalizing the original  
MFG-SIRSD
model 
(\ref{eq:SIRSD_ODEs_original_S}-\ref{eq:SIRSD_ODEs_original_D}, \ref{eq:MFG_ODEs_original_VS}-\ref{eq:MFG_ODEs_original_VR})
is to continue treating the loss of immunity as instantaneous,
while recognizing that such ``$R \rightarrow S$ transitions'' are not directly 
observed. 
(Note that we still view getting infected and recovering as observable transitions;
i.e., unlike in \cite{olmez2022modeling}, we do not consider the possibility of presymptomatic infected individuals.)
As a result, each individual does not know their immunity status and instead must form probabilistic beliefs about their immunity.
In this interpretation, $p \in [0,1]$ becomes the probability of 
still being fully immune (i.e., in $R$), $(1-p)$ is the probability of being fully 
susceptible (i.e., in $S$), and $N(\cdot,t)$ is the probability-of-immunity-structured noninfected population at the time $t$.
Thus, noninfected individuals' decisions on contact rate are based on their current personal ``immunity belief'' $p,$ and the resulting Hamilton-Jacobi PDE can be derived in the spirit of ``belief space dynamic programming'' \cite{astrom1965optimal}.
Given the average time to loss of immunity $(1/\gamma)$, it might seem reasonable at first to assume the belief dynamics of $p'(t) = -\gamma p(t),$ just as in the previous section.
However, noninfected individuals receive additional information, which modifies their belief drift: the fact that they have not become (re)-infected up till now is additional evidence that increases the probability of them still being immune. The value of this information also depends on their choice of contact rate $\cN(p(t),t)$ and the fraction of infected individuals $I(t).$  
\begin{prop*}
    Suppose a noninfected individual has last recovered at the time $t\subR$
    and now follows some chosen continuous contact rate policy $c(t)$ for $t \geq t\subR$ until a possible future re-infection. If they are fully rational and aware of the changing epidemiological situation, their immunity belief starts from $p(t\subR) =1$ and from then on follows the dynamics $p'(t) = f(p(t), c(t), t)$ with
    \begin{equation}
        \label{eq:belief_drift}
        f(p, c, t) \; = \;
        - \gamma p + \beta \cIbar I(t) c p(1-p).
    \end{equation}
\end{prop*}
\begin{proof}
We let $z(t)$ denote an individual's health state and consider\footnote{
This derivation of immunity belief dynamics is a simplified version of
M. Gee's formulation for evolving ``mode beliefs'' in piecewise-deterministic Markov processes with possible premature terminations \cite[Appendix B]{gee2025occasionally}.}
transition probabilities over a short time interval $\tau$. We assume that, for sufficiently small $\tau$, the probability of multiple transitions is negligible. The probability of a recovered individual losing immunity within a time $\tau$ is 
\begin{equation}
    \mathbb{P}(z(t + \tau) = S \ | z(t) = R) = \int_{t}^{t + \tau} \gamma e^{- \gamma(s-t)} ds .
\end{equation}

If we let $\lambda(t) = \beta c\subI c(t) I(t)$, the probability of a susceptible individual becoming infected in a time $\tau$ is given by

\begin{equation}
    \mathbb{P}(z(t + \tau) = I \ | \ z(t) = S) = \int_{t}^{t + \tau} \lambda(s) e^{- \int_{t}^{s} \lambda(\eta) d\eta} ds.
\end{equation}

We denote by $\Xi(t)$ the event that re-infection has occurred by time $t$. We can then use Bayes' theorem to find that an individual's belief at time $t + \tau$ satisfies
\begin{equation}
    \begin{aligned}
        &p(t + \tau) =\\ 
        &\frac{\mathbb{P} \big(\neg \Xi(t + \tau)\ \!\! | \!\!\ 
        z(t + \tau) \!=\! R, p(t) \big) \, \, 
        \mathbb{P} \big(z(t + \tau) \! = \! R\ \! | \! \ p(t) \big)}
        {\mathbb{P}\big(\neg \Xi(t + \tau) \ | \ p(t)\big)}.        
    \end{aligned}
\end{equation}

For small $\tau$, the probability of an individual both losing  
immunity and becoming infected during the interval $[t, t+\tau]$ is negligible, so the first factor in the numerator equals $1$. The second factor  
reflects an individual's probability of being immune after a time $\tau$ given a starting value of $p(t)$:
\begin{equation}
    \mathbb{P}(z(t+\tau)=R\mid p(t)) = p(t)(1-\gamma\tau)+O(\tau^2).
\end{equation}
The denominator is dealt with similarly, and denotes the probability of no infection occurring by the time $(t + \tau)$ given $p(t)$. We write this as
\begin{equation}
    \mathbb{P}(\neg \Xi(t + \tau) \ | \ p(t)) = 1 - (1 - p(t))(\lambda(t) \tau) + O(\tau^2).
\end{equation}
Substituting these expressions gives
\begin{equation}
p(t+\tau) = \frac{p(t)(1-\gamma\tau)+O(\tau^2)} {1-(1-p(t))\lambda(t)\tau+O(\tau^2)}.
\end{equation}
Expanding to first order in $\tau$ yields
\begin{equation}
    p(t+ \tau) = p(t) - p(t) \gamma \tau  + p(t)(1 - p(t)) \lambda(t) \tau + O(\tau^2).
\end{equation}
Subtracting $p(t)$, dividing by $\tau$, and letting $\tau \to 0$, we obtain the ODE $p'(t) = f(p(t), c(t), t)$ with the $f$ specified by \eqref{eq:belief_drift}.

\end{proof}

If the contact rate $c(t) = c$ and the level of infection $I(t) = I$ are fixed, this immunity belief ODE generally has two equilibria: $p_1^* =0$ and 
$p_2^* = 1 -\frac{\gamma}{\beta c\subI I c}.$  The larger of the two $p^* = \max(p_1^*, p_2^*)$ is the asymptotic limit (as $t\rightarrow\infty$) starting from any $p(t\subR) > 0.$  Moreover, people who remain healthy after recovery will never see their $p(t)$ decrease below $p_2^*$ as long as $\beta c\subI I c > \gamma.$ 
The fact that $p_1^* =0$ is an equilibrium 
is also convenient since it 
allows us to avoid any special handling of never-before-infected Susceptibles,
whose immunity belief will stay at $p(t) \equiv 0$.

For noninfected individuals whose probability of immunity is $p$,
if they use contact rate $c$, the expected number of them becoming infected per unit time is now $\beta (1-p) \cIbar c(t) I(t) N(p,t),$ which is precisely the same as the expression derived in the previous section.  Thus, the epidemic trajectory and the value functions still satisfy the same forward-backward system of equations
(\ref{eq:NID_N}-\ref{eq:NID_cSS}) but with a new belief drift $f(p,c,t)$ specified by \eqref{eq:belief_drift}.

\subsection{Uncertain Planning Horizons} \label{sec:uncertain}

So far, we have assumed that the end time of the epidemic is 
somehow
commonly known in advance. (E.g., if an effective vaccine is guaranteed to become broadly available by a known time $T.$) 
In reality, this is never the case, and the terminal time 
should be viewed as uncertain.
At best, there might be a commonly known \emph{a priori} probability distribution for $T$.
This complicates the contact rate optimization process for each player, since the actual realization of $T$ is revealed only later.
We note that this challenge can be treated in the framework of {\em MFGs with common noise} \cite{carmona2016mean}, particularly when the number of $T$ realizations is finite.
(E.g., this approach was used in \cite{achdou2018mean, carmona2022convergence} to model evacuation from the building under
uncertainty on which exits might be open later.) The stochastic terminal time is a very specific type of common noise closely related to optimal control problems with {\em initial uncertainty} \cite{qi2025optimality, gaspard2022optimal}, and the corresponding tree of possible scenarios is greatly simplified as a result, allowing for much more efficient numerical methods.

Here, we consider a version with $n$ potential terminal times (listed in ascending order) $T \in \mathcal{T} = \left\{T_1, \dots, T_n \right\}$ and associated probabilities $\theta_k$ satisfying $\sum_{k=1}^n \theta_k = 1$.  We define time-restricted versions of all epidemic trajectory variables and the value functions, using the superscript to indicate their time interval. I.e., $N^k, I^k, D^k, V\subN^k,$ and $V\subI^k$
are defined for $t \in [T_{k-1}, T_{k}],$ with $T_0 = 0$ to simplify the notation.  We note that the value functions now reflect the {\em expected} Nash-optimized payoff with the expectation taken with respect to $T$'s probability distribution. 
In each of these time intervals, the corresponding state variables and value functions still satisfy equations (\ref{eq:NID_N}-\ref{eq:NID_cSS})
since health status transitions are independent of the time horizon, and the Taylor series argument for the value functions holds for all $t \not \in \mathcal{T}.$  
Overall, the system of equations has to be solved over the interval $[0,T_n],$ and the previously considered deterministic $T$ case corresponds to $n=1.$
As some potential terminal time $T_k$ passes without termination,
the population fractions with different health statuses remain the same, yielding the continuity conditions
$
N^k = N^{k+1}, \,
I^k = I^{k+1}, \,
D^k = D^{k+1}
$
at each $t=T_k$ with $k=1,..., n-1.$

On the other hand, for value functions, the new information that is revealed (as the epidemic continues beyond $t=T_k$)  results in {\em jump conditions},
which are easier to explain first for the simple case of 
$\mathcal{T} = \left\{T_1, T_2 \right\}.$
If the epidemic is not over by any $t >T_1$, we know that it will deterministically stop at $T_2$, yielding the familiar terminal conditions $\VN^2(p,T_2) = \Psi\subN$ and $\VI^2(T_2) = \Psi\subI.$
On the other hand, at earlier times $t < T_1,$ the horizon is still uncertain:
with probability $\theta_1$, the epidemic ends immediately at $T_1$ (resulting in the same terminal $\Psi\subN$ and $\Psi\subI$);
otherwise, with probability $\theta_2 = (1-\theta_1),$
the epidemic continues and the payoffs over the remaining time (specified by $\VN^2$ and $\VI^2$) become relevant.  
Thus,
\begin{equation}
    \begin{aligned}
        V\subN^1(p,T_1) & = \theta_1 \Psi\subN + (1 - \theta_1) V\subN^2(p,T_1), \\
        V\subI^1(T_1) & = \theta_1 \Psi\subI + (1 - \theta_1) V\subI^2(T_1).
    \end{aligned}
\end{equation}

In a general case (with $n \geq 2$),
we still have the same terminal conditions $\VN^n(p,T_n) = \Psi\subN$ and $\VI^n(T_n) = \Psi\subI$, but the jump conditions are slightly more subtle and based on conditional termination probabilities \cite{qi2025optimality}.  
We define 
\begin{equation}
    \tilde{\theta}_k = 
    \mathbb{P} \left( T = T_k \ | \ T > T_{k-1} \right)
    = \theta_{k} / \left( \sum_{l = k}^n \theta_l \right),
\end{equation}
yielding the jump conditions for each $T_k$ ($k=1,...,n-1$): 
\begin{equation} \label{eq:jump_tk}
    \begin{aligned}
        V\subN^k(p,T_k) & = \tilde{\theta}_k \Psi\subN + (1 - \tilde{\theta}_k) V\subN^{k+1}(T_k), \\
        V\subI^{k}(T_k) & = \tilde{\theta}_k \Psi\subI + (1 - \tilde{\theta}_k) V\subI^{k+1}(T_k).
    \end{aligned}
\end{equation}

\vspace*{1mm}

\section{NUMERICAL METHODS}
\label{sec:numerics}

MFG systems are often treated by forward-backward iterations, sometimes aided by ``fictitious play'' \cite{lauriere2021numerical}. For MFGs over finite state spaces (such as our base model in section \ref{sec:basic_MFG}),
an attractive alternative is to treat the resulting ODE system by the usual numerical methods for two point boundary value problems (TPBVP), including Matlab's standard {\tt bvp5c.}  All such methods require an initial guess.  We produce one by first fixing a contact rate policy (e.g., $\cS(t) = \cNbar$), solving the epidemic trajectory via (\ref{eq:SIRSD_ODEs_original_S}-\ref{eq:SIRSD_ODEs_original_D}) forward in time, then solving  
(\ref{eq:MFG_ODEs_original_VS}-\ref{eq:MFG_ODEs_original_VR}) backward in time 
to compute the payoffs of such ``ignore the epidemic'' strategy.

We also use a similar numerical approach for our other models
after discretizing the spectrum of observably waning immunity (\S\ref{sec:full_obs_wane})
or the spectrum of confidence in full disappearing immunity (\S\ref{sec:unobs_disap}).
Instead of $p \in[0,1],$ we focus on $(m+1)$ discrete $p$-levels $p_j = jh$
with $h = 1/m$ and $j=0,\ldots, m.$
Each $p_j$ represents a cell $[\check{q}_j, \hat{q}_j]$ with
$\check{q}_j = \max( 0, \, p_j - \frac{h}{2}), \, 
\hat{q}_j = \min(1, \, p_j + \frac{h}{2}),$ and
$h_j = \hat{q}_j - \check{q}_j.$
Using cell-averaged subpopulation densities $N_j(t)  
\approx \frac{1}{h_j} \int_{\check{q}_j}^{\hat{q}_j} N(p,t) dp$
and the corresponding value functions $V_j(t) \approx \VN(p_j,t),$
we approximate each PDE involving $p$ by a system of $(m+1)$ ODEs.

To simplify the notation, we use 
$f_j = f(p_j, c_{j}^*, t),$
the standard positive/negative part notation
i.e., $f^{\pm} = \pm\max(\pm f, 0)$ with $f = f^{+} + f^{-}$,
and $\Delta$ for the Kronecker delta.
Equations \eqref{eq:NID_discr_N} yield a finite volume semi-discretization
of PDE \eqref{eq:NID_N}, 
using an upwind scheme based on flux-vector splitting\footnote{
Since $f$ is not $N$-dependent, this PDE is linear and our flux still 
ensures $N$-conservation (except for infections and recoveries, of course).
Compared to the standard Godunov flux 
$\Phi_{j+\frac{1}{2}} = f_{j+\frac{1}{2}}^+ N_{j} + f_{j+\frac{1}{2}}^- N_{j+1}$,
ours has an added advantage that with $m=1$ the system 
(\ref{eq:NID_discr_N}-\ref{eq:NID_discr_cSS})
becomes exactly equivalent
to the original SIRSD-MFG model presented in \S\ref{sec:basic_MFG}.
}
with interface fluxes
$\Phi_{j+\frac{1}{2}} = f_{j}^+ N_{j} + f_{j+1}^- N_{j+1}$
and the convention that $N_{-1} = N_{m+1} = 0$ to
avoid the special handling of edge cases.
Equations \eqref{eq:NID_discr_VN} yield a Lax-Friedrichs semi-discretization
of Hamilton-Jacobi PDE \eqref{eq:NID_VN}, 
with $\alpha \geq \max|f|$ 
and the convention that $V_{-1} = 2V_0 - V_1,$
$\, V_{m+1} = 2V_m - V_{m-1}$ to
avoid the special handling of edge cases.
The full $p$-spectrum system 
(\ref{eq:NID_N}-\ref{eq:NID_cSS})
is thus approximated by a system of $(2m+5)$ ODEs
with $j=0,\ldots, m$:
\begin{align}
    \label{eq:NID_discr_N}
    \frac{dN_j}{dt} & =  
    - \frac{\Phi_{j+\frac{1}{2}} - \Phi_{j-\frac{1}{2}}}{h}
    -\beta \cIbar I c_j^* (1 - p_j) N_j 
    + \Delta_{j,m} \mu I, \\
    \label{eq:NID_discr_I}
    \frac{dI}{dt}  & =  \beta \cIbar I \sum_{j=0}^m c_j^* (1 - p_j) N_j - \mu I - \delta I, \\
    \label{eq:NID_discr_D}
    \frac{dD}{dt} & =  \delta I, \\
     \nonumber
     \frac{d V_j}{dt} & = - \uN({c}_j^*) 
     + \beta \cIbar I {c}_j^* (1 - p_j) \left(V_j - V\subI\right) \\
    \label{eq:NID_discr_VN}
     & \hspace*{4mm} 
    - f_j \frac{V_{j+1}  - V_{j-1}}{2h}
    - \alpha \frac{V_{j+1}  - 2V_j + V_{j-1}}{2h},\\ 
    \label{eq:NID_discr_VI}
    \frac{d V\subI}{dt} & =  - \uIbar + \mu \left(V\subI - V_{m}\right) + \delta \left( V\subI - \VD\right),
\end{align}
and the Nash-optimal contact rate for noninfected individuals
\begin{align}
    \nonumber
    & \hspace*{-3mm}
    c_{j}^* =  \argmax_{c} 
    \bigg\{ 
    u\subN(c) + \beta \cIbar I c (1 - p_j) (V\subI - V_j) \, + \,\\
    \label{eq:NID_discr_cSS}    
    & 
      \hspace*{37mm}
      f(p_j,c,t) \frac{V_{j+1}  - V_{j-1}}{2h}
    \bigg\}.
\end{align}
The terminal conditions are $V_j (T) = 0$ and $\VI(T) = \Psi\subI$ while
 all $N_j(0), \, I(0),$ and $D(0)$ are specified by the initial epidemiological situation, typically with $D(0) = 0, \, N_0(0) = 1-I(0),$ and $N_j(0) = 0$ for
$j>0.$  We then solve this TPBVP using {\tt bvp5c.} 
For finer $p$-discretizations (with $m > 10$),
the resulting ODE system becomes stiff, and we require a better initial guess. To do this, we employ numerical continuation in $m$,
starting with $m = 10$ and interpolating the result as an initial guess for successively higher $m$ values\footnote{To ensure computational reproducibility, all codes used in this paper, as well as additional examples, are available at \url{https://eikonal-equation.github.io/MFG-NID-with-horizon-uncertainty}. Solving the TPBVP takes between 2.5 and 16 seconds on an Apple M3 Pro for each of the examples, depending on the level of $p$-discretization.}.  
We note that {\tt bvp5c} can also handle jump conditions at interior points, which makes it also suitable for the version of our problem with the time horizon uncertainty (\S\ref{sec:uncertain}).

\section{COMPUTATIONAL EXPERIMENTS}
\label{sec:experiments}

In this section, we demonstrate a few salient examples for each model. 
Throughout, we use the following parameter values, most of which mirror two previous MFG-studies in \cite{cho2020, buckley2025behavioral}:
$\beta = 0.05$,
$\mu = 1/10$,
$\gamma = 1/90$,
$\delta = 10^{-3}$,
and $\VD = -10^3.$
For initial conditions, we always use $D(0)=0$ and $I(0) = 5\times 10^{-3},$ with all others assumed to be initially fully susceptible with no immunity. The regularizing Lax-Friedrichs coefficient $\alpha$ must satisfy $\alpha \geq \max |f|$ to ensure stability. For the waning immunity model, we have $f(p) = - \gamma p$ so $|f(p)| \leq \gamma$ and we choose $\alpha = \gamma$. For the disappearing immunity model, the belief drift has the additional term $\beta \cIbar I c p(1-p)$ which increases the bound on $|f|$, so we use $\alpha = 1/10$ for our choice of parameters.
The planning horizon is first taken to be deterministically ten months ($T = 300$ days), and then {\em at most} ten months when we consider the time horizon uncertainty ($T_n = 300$ days).

We begin by comparing 
the basic MFG-SIRSD model (in which the loss of immunity is instantaneous and fully observed, \S\ref{sec:basic_MFG}) against 
the basic SIRSD model (\S\ref{sec:SIRSD}), in which we assume that everyone uses ``myopic'' (instantaneous utility maximizing) contact rates $\cS = \cR = \cNbar$ and $\cI = \cIbar.$  The epidemic trajectories resulting from these models are shown in Fig.~\ref{fig:SIRSD}(top).  
The myopic SIRSD model (plotted in dashed lines) exhibits a sharp increase
in Infected and a slump (once the number of Susceptibles is diminished),
followed by a small increase in Infected once many of the Recovered start to lose immunity. 
Overall, selfish/independent individuals in the MFG-SIRSD model manage to reduce the time-averaged fraction of infected by more than twofold, achieve a near twofold reduction in the peak infection level, and decrease the number of infection-related deaths by $\approx22\%.$  
This improvement of epidemic outcomes is 
achieved by reducing the contact rates of Susceptibles, $\cSS(t),$ shown in green in Fig.~\ref{fig:SIRSD}(bottom).
One subtlety worth noting is the decline in $\cSS(t)$ towards the end of the planning horizon\footnote{While this phenomenon is more prominent with our non-zero penalty for being infected at the end ($\Psi\subI,$ defined in formula \eqref{eq:infected_terminal}),
it is also present even if $\Psi\subI=0$ as long as $T$ is sufficiently large.  In the latter case, there is also an upswing in $\cSS(t) \rightarrow \cNbar$ in the last $1/\mu$ days since $\Psi\subI=0$ would reflect a different assumption that those still infected in the end are instantly cured at the time $T;$ e.g., see \cite{cho2020} and \cite{buckley2025behavioral}.}.  
This is due to a ``silver lining'' of getting infected early -- the hope of surviving and then enjoying a long period with a high contact rate while immunity lasts, even though Susceptibles still need to be careful.  
The value of this temporary bonus decreases towards the end (when the remaining time $(T-t)$ decreases below $1/\mu + 1/\gamma=100$) since we assume that Susceptibles will not have to be careful beyond the time $T.$

\begin{figure}[h]
    \centering
    \includegraphics[width=\linewidth]{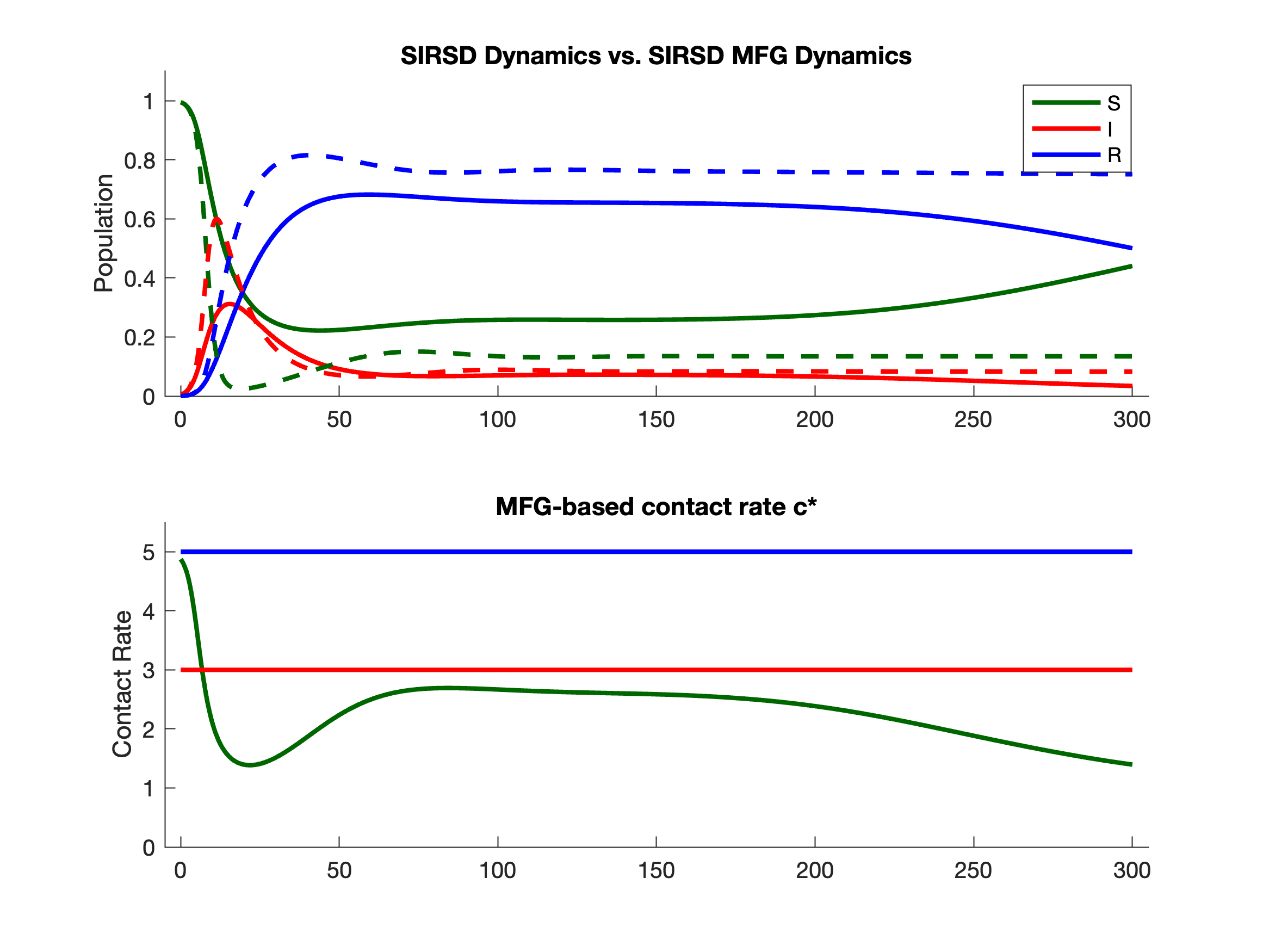}
    \caption{\small Comparison of MFG-SIRSD and ``SIRSD with myopic contact rates'' models.  TOP: The epidemic trajectories for these two models are shown by solid and dashed lines respectively. Susceptibles (S) are plotted in green, Infected (I) - in red, and Recovered/Immune (R) - in blue.  Dead (D) are not plotted to simplify the figure. Performance statistics for both models: 
    (Peak I$\,\approx0.3117,$ Mean I$\,\approx0.0823,$ Final D$\,\approx0.0247$) in MFG-SIRSD;
    (Peak I$\,\approx0.6000,$ Mean I$\,\approx0.1869,$ Final D$\,\approx0.0318$) in myopic/baseline SIRSD.
    BOTTOM: Nash-optimal contact rates in the MFG-SIRSD. 
    }
    \label{fig:SIRSD}
\end{figure}

Next, we examine the dynamics and contact rates under the fully observed waning immunity model (\S\ref{sec:full_obs_wane})
and the unobserved disappearing immunity model (\S\ref{sec:unobs_disap}). 
Since these two generalizations of MFG-SIRSD are based on very different assumptions about the disease and available information, their direct comparison of performance statistics requires many caveats, and we instead primarily compare each of them to the basic MFG-SIRSD case. 
We present these experiments in Fig.~\ref{fig:comparison} using a PDE discretization with nine different $p$-bands (i.e., $m = 8$). 

Waning immunity (Fig.~\ref{fig:comparison}(a)) makes individuals much more vulnerable soon after recovery. Thus, in this setting, a myopic/epidemic-oblivious strategy with all $c_j = \cNbar$ (not shown here due to space constraints) would lead to even more serious consequences 
(Peak I$\,\approx 0.6021,$ Mean I$\,\approx 0.2523,$ Final D$\,\approx 0.0644$) 
than in the basic SIRSD case of Fig.\ref{fig:SIRSD}.
Nash-optimizing agents respond to this increased vulnerability by adopting
much lower contact rates in the lower (more susceptible) $p$-bands,
leading to an almost twofold decrease in peak infection compared to MFG-SIRSD even though the time-averaged infection level and the mortality become slightly higher.  
In the highest immunity band, $p \in [\frac{15}{16}, 1],$ the lack of risk in socializing yields $c_m^*(t) = \cNbar.$  
In the next immunity band,  $p \in [\frac{13}{16}, \frac{15}{16}),$ the contact rate is already only slightly above the Infected's $\cIbar=3.$
In the third highest band, $p \in [\frac{11}{16}, \frac{13}{16}),$
the Nash-optimal contact rate dynamics is closer to what Susceptibles did in Fig.~\ref{fig:SIRSD}, but with a less cautious behavior near the peak of infection and no reduction in contact rates as $t \rightarrow T.$  

\begin{figure*}[!t]
    \centering
    \begin{subfigure}[t]{0.48 \textwidth}
        \centering
        \includegraphics[width=\linewidth]{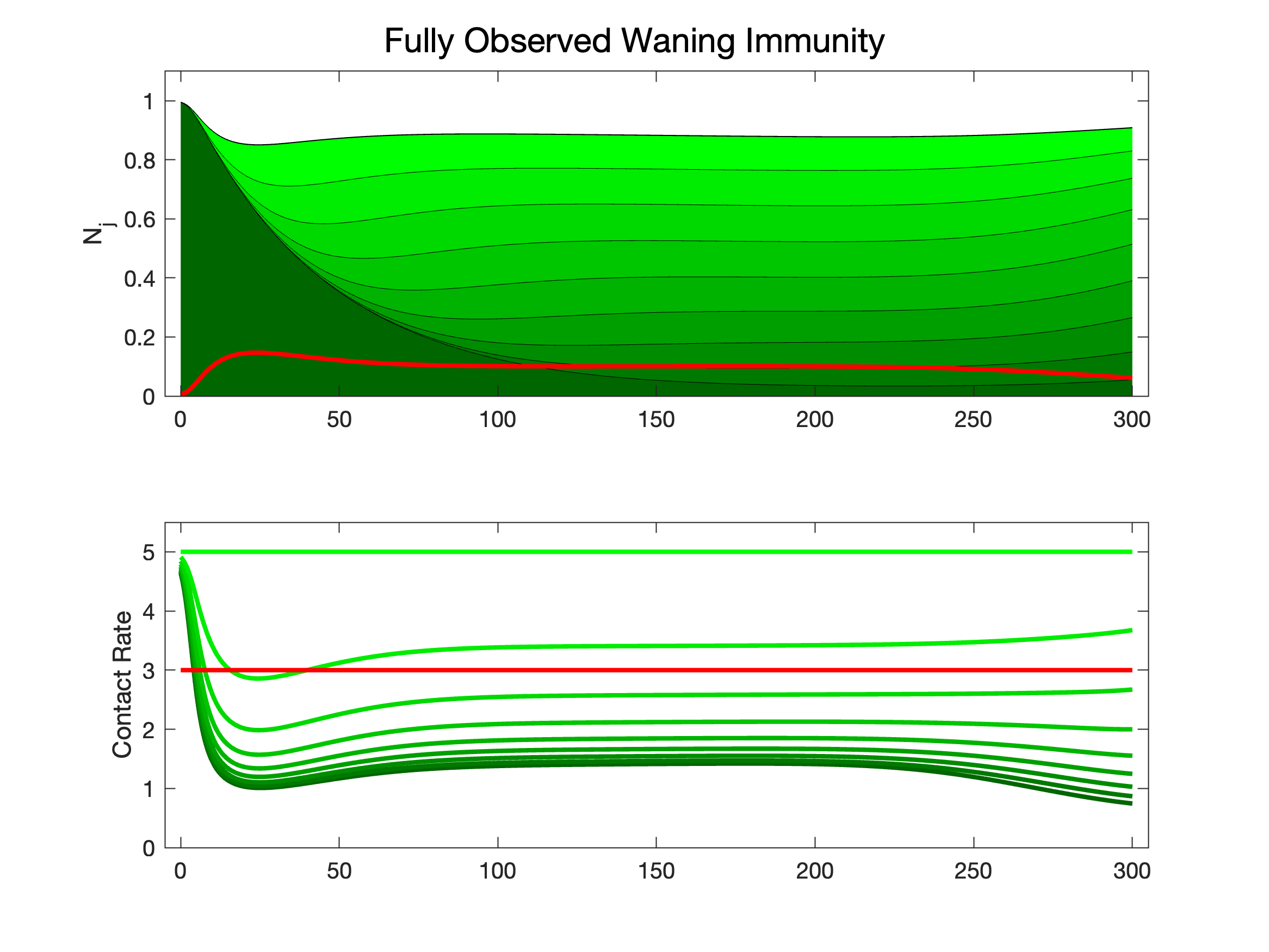}
        \caption{\small Fully observed waning immunity. Performance statistics:\\ (Peak I $\approx$ 0.1468, Mean I $\approx$ 0.1002, Final D $\approx$ 0.0301).}
        \label{fig:wane}
    \end{subfigure}
    \vspace{0.5cm} 
    \begin{subfigure}[t]{0.48 \textwidth}
        \centering
        \includegraphics[width=\linewidth]{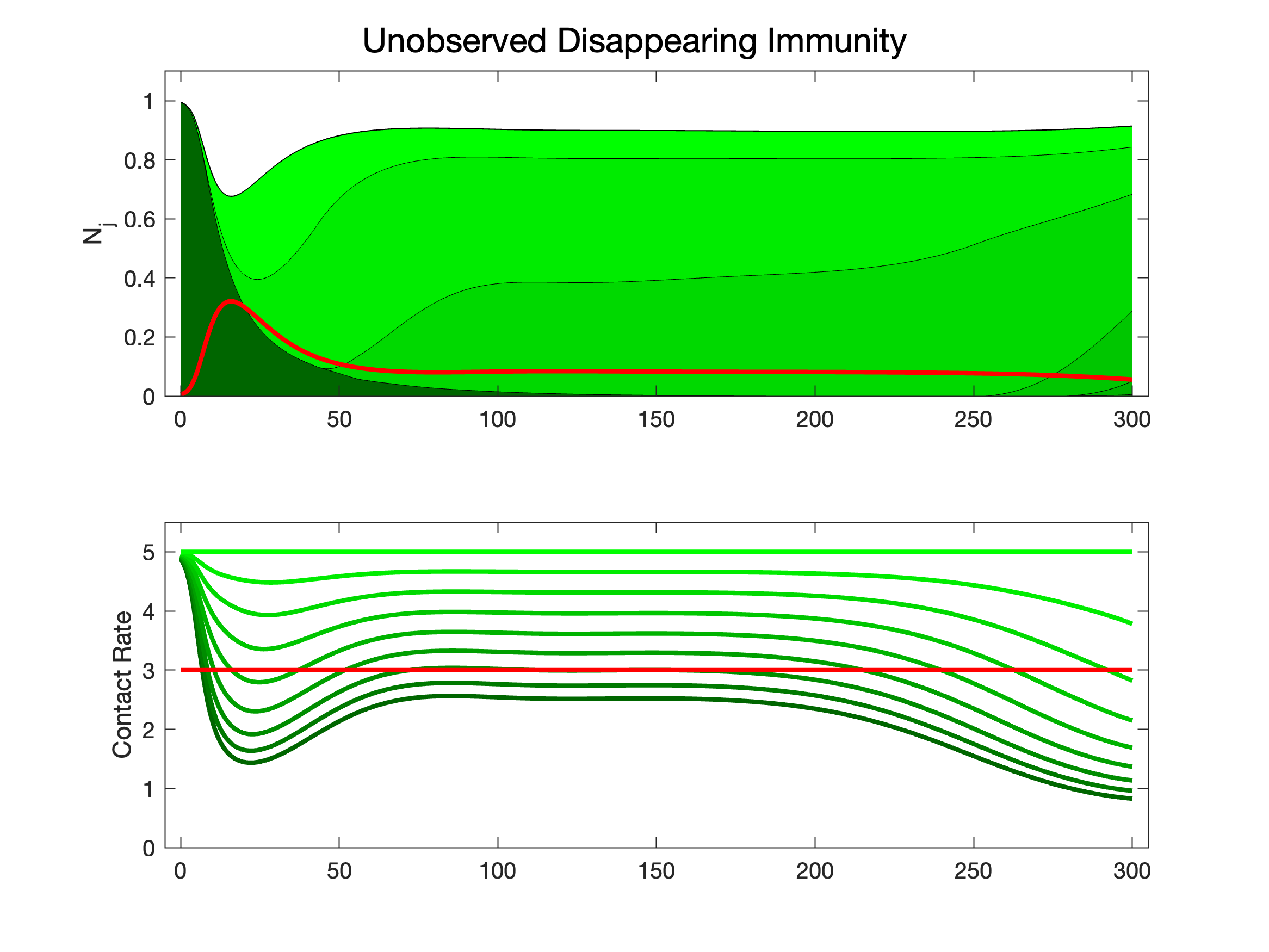}
        \caption{\small Unobserved disappearing immunity.
        Performance statistics: (Peak I $\approx$ 0.3209, Mean I $\approx$ 0.0985, Final D $\approx$ 0.0296).}
        \label{fig:bayes}
    \end{subfigure}
    \caption{\small Dynamics and contact rates under (a) the fully observed waning immunity model and (b) the unobserved disappearing immunity model;
    both computed for nine $p$-bands (i.e., $m=8$).
    TOP: Fractions of noninfected individuals $N_j(t)$ shown by stacked green area plots, with the shades of green corresponding to different $p$-bands  (interpreted as different bands of immunity in subfigure (a) on the left vs.  different bands of immunity {\em confidence} in subfigure (b) on the right).  The darkest green corresponds to the highest susceptibility ($p \in [0,\frac{1}{16})$); the lightest green is the highest immunity ($p \in [\frac{15}{16}, 1]$).  The fraction of infected $I(t)$ is plotted on top in red.  
    BOTTOM: Nash-optimal contact rates for each $p$-band in each model.
    \label{fig:comparison}
    }
\end{figure*}

Turning to the unobserved disappearing immunity model (Fig.\ref{fig:comparison}(b)), we see a slight worsening in all three performance statistics compared to the basic MFG-SIRSD model (Fig.~\ref{fig:SIRSD}). 
This is not surprising, since the loss of immunity was fully observable there, so each individual had more information when choosing contact rates.
Those who are fairly sure that they do not have immunity ($p \in [0,\frac{1}{16})$, the darkest green in Fig.\ref{fig:comparison}(b), including those who have never been infected) are somewhat more cautious than the Susceptibles in Fig.~\ref{fig:SIRSD}. However, at higher immunity confidence levels, the contact rates are significantly higher than we have seen in other models, approaching $\cNbar$ at a fairly uniform rate as $p \rightarrow 1.$  Moreover, unlike in the case of waning immunity, here those who do not become re-infected for a long time stay more confident about their chances of remaining fully immune. On the time interval $[75,175],$ the infection level and contact rates stay fairly constant, and post-recovery immunity confidence levels do not decrease below 
$p^* = 0.7926$. 
As a result, the majority of noninfected individuals fall into two confidence bands spanning $p \in [\frac{11}{16}, \frac{15}{16}),$ and they adopt far less conservative contact rates $c_7^*(t)$ and $c_6^*(t).$ 
Quite a few of them eventually become re-infected and later re-enter the highest immunity confidence band $p \in [\frac{15}{16}, 1]$ upon recovery.

To demonstrate the effects of time-horizon uncertainty (\S\ref{sec:uncertain}), we first look at two potential terminal times $T \in \{150, 300\}$ and 
assume that the epidemic might end after 5 months (at $t=150$) with 
probability $\theta.$  Focusing for now on MFG-SIRSD, in Fig.~\ref{fig:sweep} we show the consequences of this uncertainty for three different likelihood levels $\theta = 0.1, 0.5, 0.9.$
In all cases, Susceptibles' Nash-optimal contact rates 
decline as we approach that possible early termination (for reasons already explained in the MFG-SIRSD analysis), followed by a discontinuous jump to a higher level (temporarily overcompensating) 
if the epidemic turns out not to end at $t=150.$
This jump is due to two factors: (1) the ``silver lining of infection'' argument now applies again for a large stretch 
of the remaining 5 months\footnote{
Declining compliance with social distancing recommendations
has been also observed in practice, when people realize that 
an ongoing
epidemic will last significantly longer than originally expected \cite{petherick2021worldwide}.  While it is natural to attribute such behavior to compliance fatigue, the above ``silver lining of infection'' argument might be 
also a part of their rationalization.
}
 and (2) the chances of becoming ill in the near future
are somewhat lower since the fraction of Infected 
has been decreasing on the way to $t=150.$
Not surprisingly, the size of the jump in contact rates 
increases with $\theta$;  the original deterministic MFG-SIRSD version can be recovered in the limit as $\theta \rightarrow 0.$ 

\begin{figure}[h]
    \centering
    \includegraphics[width=\linewidth]{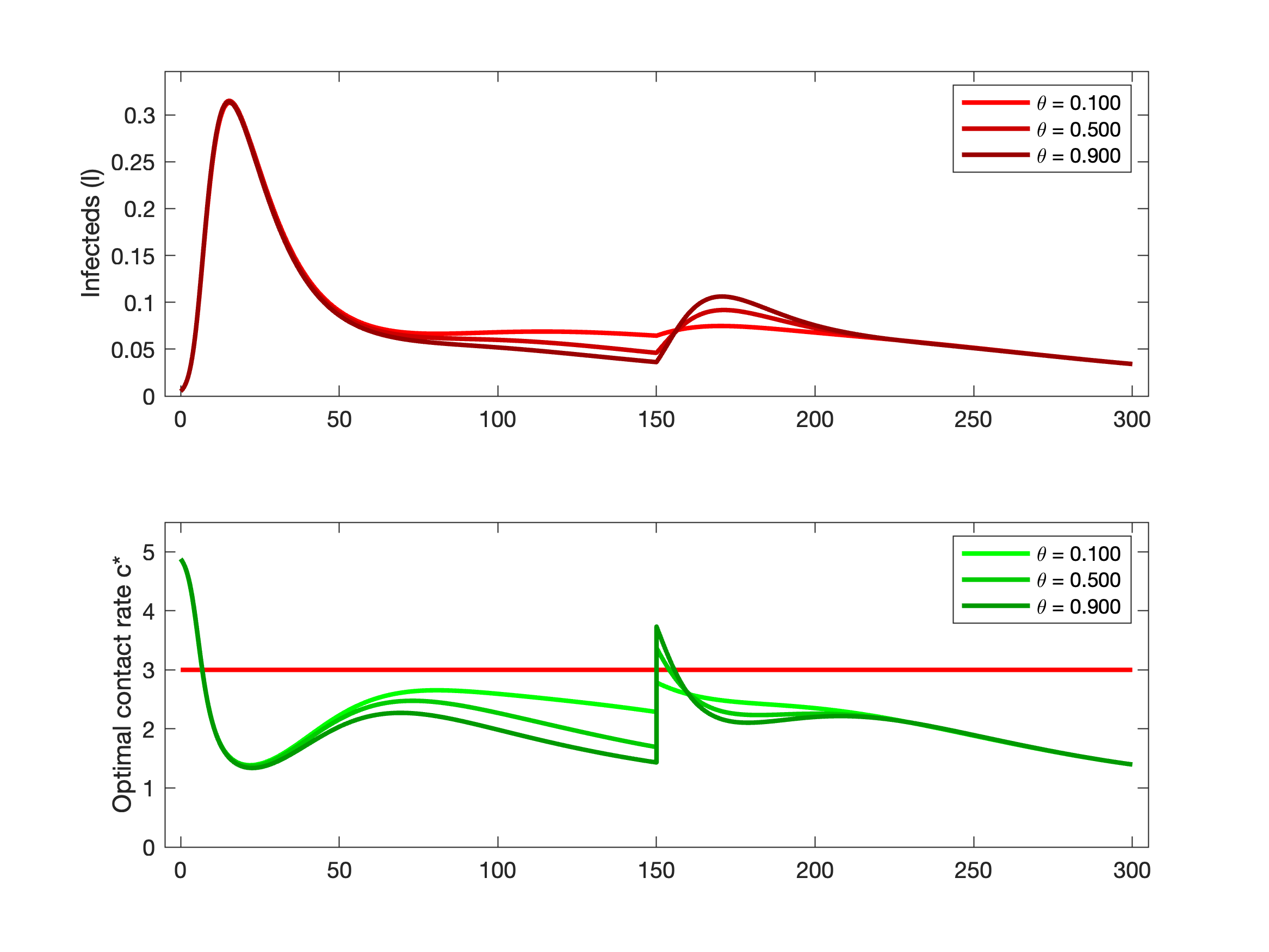}
    \caption{\small Uncertain horizon MFG-SIRSD example with $T \in \{150, 300\}$
    and $\mathbb{P}(T=150) = \theta,$ shown for three different values of $\theta.$  TOP: Infection dynamics.  BOTTOM: Nash-optimal contact rates for Susceptibles.}
    \label{fig:sweep}
\end{figure}

Our framework can be similarly used to handle horizon-uncertainty
with a larger number of possible $T$ values and with more complicated
immunity models. 
Due to space constraints, we include only one representative example (Fig.~\ref{fig:5jumps_bayes}) with five possible terminal times of the epidemic in an unobserved disappearing immunity model. 
To make this figure easier to interpret, we use only five $p$-bands ($m=4$) and one specific probability distribution on possible terminal times. 
Qualitatively, the observed features are consistent with those already explained above for Fig.~\ref{fig:sweep}.

\begin{figure}[h]
    \centering
    \includegraphics[width=\linewidth]{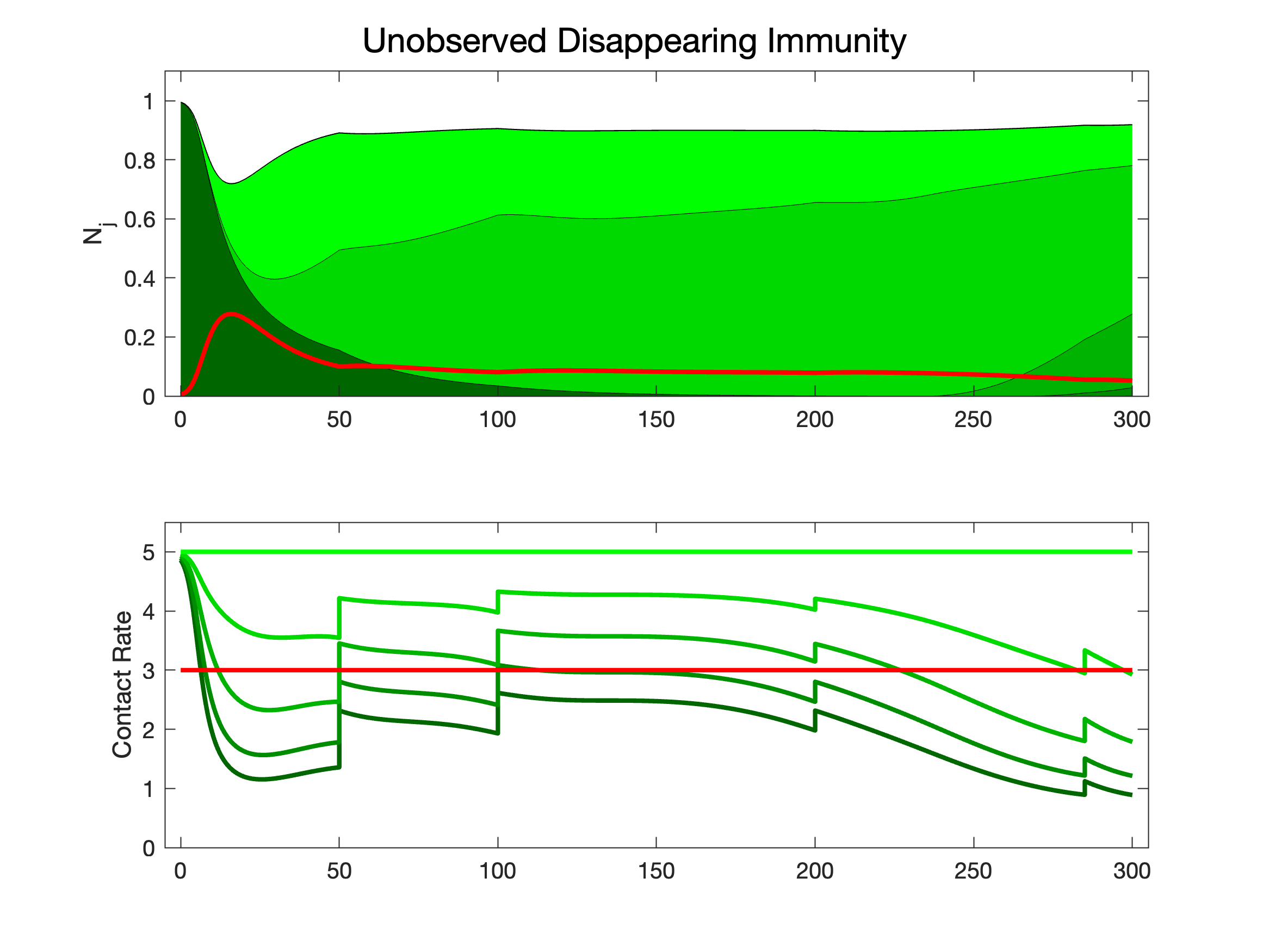}
    \caption{\small Horizon-uncertainty implemented in the unobserved disappearing immunity model.  
    Terminal $T \in \{50, 100, 200, 285, 300\}$ with the associated probabilities $(0.2, 0.1, 0.05, 0.5, 0.15).$ Computed and plotted with only five $p$-bands ($m=4$) for the sake of visual interpretability.
    TOP: dynamics of Infected (red) and Noninfected (green) fractions,
    with the latter shown using stacked area plots and shades of green corresponding to different immunity confidence bands.
    Performance statistics: (Peak I $\approx$ 0.2782, Mean I $\approx$ 0.1156, Final D $\approx$ 0.0222).
    BOTTOM: Nash-optimal contact rates for Noninfected, exhibiting jump discontinuities at all possible ``early termination'' times.}
    \label{fig:5jumps_bayes}
\end{figure}

\section{CONCLUSIONS}
\label{sec:conclusions}
We presented a modeling and computational framework for
incorporating partial observability and time-horizon uncertainty
in epidemiological MFG models.  We hope that similar approaches
will be also useful in other applications of MFGs.
To increase its practical applicability, our model in \S\ref{sec:full_obs_wane} could be 
extended to handle immunity-dependent course of infection 
(e.g., reduced utility penalties and a lower death rate for those re-infected relatively soon after recovery \cite{angelov2024immuno}).
Both models will also become more realistic once we include
subpopulations with other (non-rational) behavioral patterns \cite{buckley2025behavioral} 
or with presymptomatic individuals \cite{olmez2022modeling}.
Additional directions for future work include 
hybrid (waning/disappearing) immunity models
and continuous distributions for the uncertain time horizon.

\newest{In this paper, we have primarily focused on the modeling framework and numerics, sidestepping important theoretical questions.  E.g., while all of our numerical tests indicate the well-posedness of TPBVP problem (\ref{eq:NID_discr_N}-\ref{eq:NID_discr_cSS}), it will be highly useful to prove the uniqueness of Nash equilibrium rigorously both in that semi-discretized ODE version and in the full forward-backward PDE system (\ref{eq:NID_N}-\ref{eq:NID_cSS}).}

%FROM IEEE template: A conclusion section is not required. Although a conclusion may review the main points of the paper, do not replicate the abstract as the conclusion. A conclusion might elaborate on the importance of the work or suggest applications and extensions. 


%%%%%%%%%%%%%%%%%%%%%%%%%%%%%%%%%%%%%%%%%%%%%%%%%%%%%%%%%%%%%%%%%%%%%%%%%%%%%%%%

%%%%%%%%%%%%%%%%%%%%%%%%%%%%%%%%%%%%%%%%%%%%%%%%%%%%%%%%%%%%%%%%%%%%%%%%%%%%%%%%



% \section*{ACKNOWLEDGMENT}

% We would like to thank...

%%%%%%%%%%%%%%%%%%%%%%%%%%%%%%%%%%%%%%%%%%%%%%%%%%%%%%%%%%%%%%%%%%%%%%%%%%%%%%%%

\bibliographystyle{IEEEtran}
\bibliography{IEEEabrv,references}

\end{document}